\documentclass{amsart}
\usepackage{amsmath}
\usepackage{paralist}
\usepackage[pagewise]{lineno}

\usepackage[numbered]{bookmark}
\usepackage{graphicx}
\usepackage{amsfonts}
\usepackage{amsthm} 
\usepackage{amssymb}
\usepackage{cite}
\usepackage{mathrsfs}
\usepackage{changepage} 
\usepackage{mathtools}
\usepackage{mathabx}
\usepackage{verbatim} 
\usepackage{enumerate}
\usepackage{xfrac}

\hypersetup{colorlinks=true, urlcolor=blue, citecolor=red}

\newtheorem{thm}{Theorem}[section]

\newtheorem{lem}[thm]{Lemma}
\newtheorem{prop}[thm]{Proposition}

\newtheorem*{claim}{Claim}

\theoremstyle{remark}
\newtheorem*{remark}{Remark}

\newcommand{\R}{\mathbb{R}}
\newcommand{\Z}{\mathbb{Z}} 
\newcommand{\T}{\mathbb{T}} 
\newcommand{\N}{\mathbb{N}}

\begin{document}

\title[Lyapunov instability in two degrees of freedom]{Lyapunov instability in KAM stable Hamiltonians with two degrees of freedom}
\author[Frank Trujillo]{}
\email{frank.trujillo@imj-prg.fr}
\maketitle

\centerline{\scshape Frank Trujillo}
\medskip
{\footnotesize
 \centerline{IMJ-PRG, Universit\'e de Paris}
 \centerline{Paris, France}
}

\bigskip

%\centerline{(Communicated by the associate editor name)}

\begin{abstract}
For a fixed frequency vector $\omega \in \mathbb{R}^2 \, \setminus \, \lbrace 0 \rbrace$ obeying $\omega_1 \omega_2 < 0$ we show the existence of Gevrey-smooth Hamiltonians, arbitrarily close to an integrable Kolmogorov non-degenerate analytic Hamiltonian, having a Lyapunov unstable elliptic equilibrium with frequency $\omega$. In particular, the elliptic fixed points thus constructed will be KAM stable, i.e. accumulated by invariant tori whose Lebesgue density tend to one in the neighbourhood of the point and whose frequencies cover a set of positive measure. 

Similar examples for near-integrable Hamiltonians in action-angle coordinates in the neighbourhood of a Lagragian invariant torus with arbitrary rotation vector are also given in this work.
\end{abstract}

\maketitle

\section{Introduction}
\label{sec: introduction}

By classical KAM theory, non-resonant elliptic fixed points satisfying the Kolmogorov's non-degeneracy condition are accumulated by invariant tori whose Le-besgue density tend to one in the neighbourhood of the point and whose frequencies cover a set of positive measure. This property is sometimes called \textit{KAM stability.} V. Arnold \cite{arnold_stability_1961} proved\footnote{According to Bruno \cite{bruno_question_1989} Arnold's original proof was not correct. A different proof of this result can be found in \cite{siegel_lectures_2012}.} 
that non-resonant elliptic fixed points of Hamiltonians with two degrees of freedom which are simultaneously Kolmogorov and isoenergetic non-degenerate must be Lyapunov stable. 

Restricted to elliptic fixed points in two degree-of-freedom Hamiltonian systems, we consider the question of wether or not KAM stability by itself implies stability in the sense of Lyapunov. We point out that for systems with three or more degrees of freedom this question has a negative answer. In fact, examples related to Arnold's diffusion \cite{arnold_instability_1964} show that for systems with at least five degrees of freedom KAM-stability does not imply Lyapunov stability of the elliptic equilibrium, even if we assume isoenergetic non-degeneracy. More recently, B. Fayad \cite{fayad_lyapunov_2019} constructed similar examples for systems with three or more degrees of freedom.

In this work we show that for systems with two degrees of freedom the answer to this question is also negative. Moreover, our examples can be taken arbitrarily close to non-degenerate \textit{integrable} Hamiltonians, i.e. systems for which the phase space is completely foliated by invariant tori. The construction is inspired in the diffusion along resonances mechanism for Hamiltonians in action-angle coordinates, that we describe in detail in Section \ref{example}. Let us point out that similar results to those concerning the stability and accumulation by invariant tori of non-resonant elliptic fixed points also hold for Diophantine invariant tori of Hamiltonians in action-angle coordinates. The aforementioned mechanism will allow us to construct examples of Lyapunov unstable invariant tori with arbitrary rotation vector for non-degenerate near-integrable Gevrey-Smooth Hamiltonians in action-angle coordinates. As we shall see, for frequency vectors $\omega = (\omega_1, \omega_2)$ obeying $\omega_1\omega_2 < 0$, this construction can be carried to the elliptic fixed point case. 

\section{Setting and notations}
\label{UE: sc_setting}
\subsection{Hamiltonians flows near a fixed point}

Throughout this work we denote vectors in $\R^4$ by $(x_1,y_1,x_2,y_2)$ and endow $\R^4$ with its canonical symplectic structure $dx_1\wedge dy_1+dx_2 \wedge dy_2$. Recall that a Hamiltonian on $\R^4$ is simply a smooth function $H$ on $\R^4$ and that its associated \textit{Hamiltonian system} is given by
\begin{equation}
\label{eq: hamiltonian_system}
\dot{x_i} = \partial_{y_i} H, \hskip0.5cm \dot{y_i} = -\partial_{x_i} H.
\end{equation}
The solution to the previous system, which we denote by $\Phi_H^t$, is called the \textit{Hamiltonian flow} of $H$. We denote by $X_H$ the \textit{Hamiltonian vector field} given by (\ref{eq: hamiltonian_system}). A fixed point $p_0$ of a Hamiltonian flow $H$ is said to be \textit{Lyapunov unstable} if there exists $A > 0$ such that for any $\delta > 0$ there exist $p$ and $t \in \R$ obeying 
\[|p - p_0| < \delta, \hskip1cm \left| \Phi^t_{H}(p) - p_0\right| > A. \]
A fixed point of a Hamiltonian flow is called $\textit{elliptic}$ if the eigenvalues of the linearized system are purely imaginary. In this case the spectrum of the associated matrix will be given by $\lbrace \pm i\omega_1, \pm i \omega_2 \rbrace$ for some $\omega = (\omega_1, \omega_2)$ in $\R^2$. We call $\omega$ the \textit{frequency vector} of the elliptic equilibrium. A vector $\omega \in \R^2$ is said to be \textit{non-resonant} if the equation $\langle \omega, k \rangle = 0$ has no solution $k \in \Z^2 \,\setminus\, \{0\}.$ An elliptic fixed point is said to be non-resonant if its frequency vector is non-resonant. 

Every Hamiltonian $H:\R^4 \rightarrow \R$ having a non-resonant elliptic fixed point at the origin with frequency vector $\omega$ can be expressed in the form 
\begin{equation}
\label{ellipticeq}
H(x,y) = \sum_{i=1}^2 \dfrac{x_i^2+y_i^2}{2} \omega_i + \mathcal{O}^3(x,y)
\end{equation}
by means of a linear symplectic transformation (see \cite{arnold_mathematical_2007} for a proof). Moreover, if we define $I = (I_1,I_2): \R^4 \rightarrow \R^2$ by 
 \[ I_i (x, y) = \dfrac{x_i^2+y_i^2}{2},\]
for any $n \in \N$ the Birkhoff normal form ensures the existence of a symplectic transformation taking the Hamiltonian (\ref{ellipticeq}) to the form
 \[ H(x, y) = f_n(I(x, y)) + \mathcal{O}^{2n + 1}(x, y),\]
where $f_n$ is a polynomial of degree at most $n.$ A smooth Hamiltonian $H$ of the form 
\begin{equation}
\label{BNF}
H(x, y) = h(I(x, y)) + \mathcal{O}^5(x, y)
\end{equation}
is said to be \textit{Kolmogorov non-degenerate} if 
\begin{equation}
\label{eq: kolmogorov_non_degenerate}
\det \left| \partial_I^2 h(0) \right| \neq 0
\end{equation}
and is said to be \textit{isoenergetic non-degenerate} if 
\begin{equation}
\label{eq: isoenergetic_non_degenerate}
\det \left| \begin{array}{ll} \partial_I^2 h(0)& \partial_I h (0)^T \\ 
\partial_I h(0) & \hspace{0.4cm} 0 \end{array} \right| \neq 0 .
\end{equation}
Notice that for $H$ of the form (\ref{BNF}) the origin is an elliptic equilibrium with rotation vector $\partial_I h(0).$ In the following, given a smooth function $h: U \subset \R^2 \rightarrow \R$ defined on an open neighbourhood of $0$ we will say that $h$ is \textit{Kolmogorov} (resp. \textit{iso-energetic}) \textit{non-degenerate} if (\ref{eq: kolmogorov_non_degenerate}) (resp. (\ref{eq: isoenergetic_non_degenerate})) is verified. 
 
A smooth Hamiltonian $H$ of the form (\ref{BNF}) is said to be \textit{integrable} if it can be conjugated by means of a symplectic transformation to a Hamiltonian of the form $h(I(x, y))$. Notice that in this case the phase space will be completely foliated by invariant tori. 

\subsection{Hamiltonians in action-angle coordinates} In the following we denote vectors in $\T^2 \times \R^2$ as $(\theta, R)= (\theta_1, \theta_2, R_1, R_2)$ and we endow $\T^2 \times \R^2$ with the canonical symplectic structure $d\theta_1 \wedge dR_1 + d\theta_2 \wedge dR_2.$ We will refer to $I$ and $\theta$ as the \textit{action} and \textit{angle} coordinates respectively. A Hamiltonian defined on an open set of the form $\T^2 \times U \subset \T^2 \times \R^2$ is said to be in action-angle coordinates. Recall that by the Arnold-Liouville theorem under mild non-degeneracy hypotheses near-integrable Hamiltonians can put into action-angle coordinates in the neighbourhood of an invariant torus of the unperturbed system.

For a Hamiltonian $H$ on $\T^2 \times \R^2$ its associated Hamiltonian system is given by
\[ \dot{\theta_i} = \partial_{R_i} H, \hskip0.5cm \dot{R_i} = -\partial_{\theta_i} H.\]
In this work we will be mostly interested in Hamiltonians on $\T^2 \times \R^2$ for which the torus $\T_0 = \T^2 \times \lbrace 0 \rbrace$ is invariant and whose restricted dynamics are given by a continuous translation, that is, smooth Hamiltonians of the form
\begin{equation}
\label{quasiper}
H (\theta,R) = h(R) + \mathcal{O}^2(R).
\end{equation}
In this case, denoting $\omega = \partial_R h (0)$, the associated Hamiltonian flow $\Phi^t_H$ restricted to $\T_0$ is given by 
\[\Phi^t_H(\theta, 0) = (\theta + t \omega, 0)\]
and we say that the invariant torus $\T_0$ has \textit{translation vector} $\omega.$ The invariant torus $\T_0$ is said to be \textit{Lyapunov unstable} if there exists $A > 0$ such for every $\delta > 0$ there exists $p$ and $t \in \R$ obeying
\[d(p,\T_0) < \delta, \hskip1cm d\left(\Phi^t_{H}(p), \T_0\right) > A. \]
A Hamiltonian $H$ of the form (\ref{quasiper}) is said to be \textit{Kolmogorov} (resp. \textit{iso-energetic}) \textit{non-degenerate} if $h$ is Kolmogorov (resp. iso-energetic) non-degenerate. 

As we shall see, for $h: U \subset \R^2 \rightarrow \R$ defined on an open neighbourhood of the origin, some of the elliptic equilibrium's stability properties for Hamiltonians of type (\ref{BNF}) can be related to stability properties of the torus $\T^2 \times \{0\}$ for Hamiltonians in action-angle coordinates of the form (\ref{quasiper}). 

\subsection{Function spaces} Since the properties we are interested in are purely local, let us introduce suitable spaces of functions that will simplify the exposition. For $s > 1$ let $\mathcal{G}^{s}(\mathbb{T}_0)$ denote the set of Gevrey-smooth functions of class $s$ defined on some open neighbourhood of $\T_0 \subset \T^2 \times \R^2$. Recall that a real $C^\infty$ function $f$ defined on an open set $U$ is said to be \textit{Gevrey of class $s$} if for every compact set $K \subset U$ there exist positive constants $c, \rho$ such that 
\[ \| f\|_K = \sup_{z \in K} |\partial ^ \alpha f(z)| \leq c\rho^{|\alpha|}(\alpha!)^s\]
for every multi-index $\alpha.$ Similarly, for $r = \infty, \omega$ and for any $s > 1$ we denote by $\mathscr{C}^r(0)$ and $\mathcal{G}^{s}(0)$ the set of $C^r$ and of Gevrey-smooth functions of class $s$ defined on some open neighbourhood of $0 \in \R^2$. 

\section{Instability in action-angle coordinates}
\label{sec: torus_instability}

The goal of this section is to prove the following.

\begin{thm}
\label{thm: actionangle}
Given $\sigma > 0$ and $\omega \in \R^2 \,\setminus\,\lbrace 0 \rbrace$ there exist a Kolmogorov non-degenerate analytic function $h \in \mathscr{C}^\omega(0)$ and a Gevrey-smooth function $f \in \mathcal{G}^{1+\sigma}(\mathbb{T}_0)$ flat at $\T_0$ such that for any $\epsilon > 0$ $\T_0$ is a Lyapunov unstable invariant torus with rotation vector $\omega$ for the Hamiltonian 
 \[H_\epsilon(\theta, R) = h(R) + \epsilon f(\theta, R).\]
\end{thm}

\begin{remark}
The unstable orbits exhibited in the proof of the theorem will drift away from $\T_0$ at a linear speed. See (\ref{eq: proof_diffusion}).
\end{remark}

This theorem is proven in Section \ref{sketch}. To motivate its proof, let us give first a detailed explanation of the diffusion mechanism we intend to use.

\subsection{Diffusion along resonances}
\label{example}

For the remaining of this work, given a vector $v = (v_1, v_2)\in \R^2$ we denote 
\[|v| = |v_1| + |v_2|, \hskip1cm |v|_{\max} = \max\{|v_1|, |v_2|\} \] 
\[ v^\perp = (-v_2, v_1), \hskip1cm \langle v \rangle = \textup{Vect}(v).\]
Let us start by considering a simple example of a near-integrable system admitting unbounded orbits. Suppose $k \in \Z^2$ and define $h : \R^2 \rightarrow \R$ as $h(R) = \langle k, R\rangle.$ Notice that $k$ is resonant since $\langle k, k^\perp \rangle = 0$. Fix $\epsilon \in \R$ and define $H_\epsilon : \T^2 \times \R^2 \rightarrow \R$ as 
\begin{equation}
\label{simpleinstability}
 H_\epsilon(\theta, R) = h(R) + \epsilon \cos(\theta \cdot k^{\bot}).
 \end{equation}
The Hamiltonian flow associated to $H_\epsilon$ is given by 
\[ \Phi_{H_\epsilon}^t (\theta,R) = (\theta + tk, R+ t\epsilon \sin(\theta \cdot k^{\bot})k^{\bot}).\]
Let $\theta_0 \in \T^2$ such that $\sin(\theta_0 \cdot k^{\bot}) = 1$. Then for any $R_0 \in \R^2$
\[ \Phi_{H_\epsilon}^t(\theta_0, R_0) = (\theta_0 + tk, R_0 + t\epsilon k^{\bot})\]
defines an unbounded solution for the Hamiltonian which drifts at linear speed along the affine subspace $\Lambda = R_0 + \langle k^\bot \rangle$. 

Although in the previous example $\T_0$ is not invariant by $H_\epsilon$, this can be easily fixed if we multiply the perturbative term by an appropriate bump function. Nevertheless, $h_0$ is far from being Kolmogorov non-degenerate. Moreover, although $\T_0$ is invariant for $h_0$, for the previous construction to make sense its rotation vector has to be an integer vector. We claim that the main feature allowing solutions to `escape' is not the particular form of the function $h$ but the fact that the gradient of $h$ restricted to some affine subspace $\Lambda$ of rational slope remains orthogonal to it. This will allow us to adapt the previous construction and to consider similar examples for more general integrable Hamiltonians $h.$
 
In fact, let $h: \R^2 \rightarrow \R$ smooth and $\Lambda \subset \R^2$ be a line (not necessarily passing through the origin) of rational slope such that $h\mid_\Lambda$ is constant. Then $\Lambda = R_0 + \langle k^\perp\rangle,$ for some $R_0 \in \R^2$ and $k \in \Z^2.$ Notice that $h\mid_\Lambda$ being constant is equivalent to $\partial_Rh$ being orthogonal to $\Lambda$. Let $H_\epsilon$ as in (\ref{simpleinstability}), $\theta_0 \in \T^2$ such that $\sin(\theta_0 \cdot k^{\bot}) = 1$ and denote 
\[ z(t) = (z_\theta(t), z_R(t)) = \Phi_{H_\epsilon}^t(\theta_0, R_0).\]
We claim that $z_R(t)$ verifies $z_R(t) = R_0 + t\epsilon k^\perp$ for all $t \in \R.$ Recall that $\Phi_{H_\epsilon}^t$ is the flow associated to the Hamiltonian vector field
\begin{align*}
X_{H_\epsilon}(\theta, R) & = (\partial_R H_{\epsilon} (\theta, R), -\partial_\theta H_{\epsilon}(\theta, R)) \\
& = (\partial_R h(R), \epsilon \sin(\theta \cdot k^\perp) k^\perp).
\end{align*}
Since $\partial_\theta H_\epsilon$ is always collinear with $k^\perp$ it follows that $z_R(t) - R_0$ is collinear with $k^\perp$ for all $t \in \R$. Thus $z(t) \in \T^2 \times \Lambda$ for all $t \in \R.$ Since $\partial_R h \mid_{\Lambda}$ is always collinear with $k$ it follows that $ \langle z_\theta(t) - \theta_0 , k^\perp \rangle = 0$ for all $t \in \R.$ Hence
\[ \dot{z}(t) = X_{H_\epsilon}(z(t)) = (\partial_R h (z_R(t)), \epsilon k^\perp).\]
Therefore $z_R(t) = R_0 + t \epsilon k^\perp$ for all $t \in \R.$

This obstruction to stability (the fact that the gradient of the unperturbed Hamiltonian when restricted to some linear subspace might remain orthogonal to it) its well-known. Moreover, N. Nekhoroshev \cite{nekhoroshev_exponential_1977} showed that this was the main obstruction for the \textit{exponential stability} of invariant tori, where by exponential stability we mean that solutions close to the invariant torus remain close to it for an interval of time that is exponentially large with respect to the inverse of the distance to the torus. Nekhoroshev introduced the concept of \textit{steep functions} which quantifies the idea of the gradient of a function not remaining orthogonal to any given linear subspace and proved that generically invariant torus of analytic near-integrable Hamiltonians having a Diophantine rotation vector are exponentially stable. Recent works by A. Bounemoura et al. \cite{bounemoura_superexponential_2017} improved Nekhoroshev's result and showed that generically Diophantine invariant torus of analytic or Gevrey smooth near-integrable Hamiltonians are actually super-exponentially stable. 

\subsection{Proof of Theorem \ref{thm: actionangle}}
\label{sketch}

Let us first give a brief sketch of the proof. Inspired by the examples considered in the previous section, we will construct an integrable Hamiltonian $h$ defined on an open neighbourhood of $0$ and whose gradient is orthogonal to a family of affine subspaces approaching the origin. We then define an appropriate perturbation $f$ in order to obtain unbounded solutions for the Hamiltonian $H_\epsilon$ with initial conditions arbitrarily close to $\T_0$. The perturbation will be flat on $\T_0$ and hence $\T_0$ will be invariant for the perturbed Hamiltonian. 

Fix $\omega $ in $\R^2 \,\setminus\, \lbrace 0 \rbrace.$ Let $v: \R \rightarrow \R^2$ analytic with $v(0) = \omega$ and denote 
\[\Lambda_t = (0,t) + \langle v^\perp(t) \rangle\]
for every $t \in \R$. Suppose there exists $h :\R^2 \rightarrow \R$ analytic such that
\begin{equation}
\label{gradyaxis}
\nabla h(0,t) = v(t), \hskip1cm h\mid_{\Lambda_t} = h(0,t),
\end{equation} 
and assume that the graph of $v$ is not contained in any line of the plane. Since close to the origin the angle of $v$ with the $x$-axis changes we can find sequences $(y_n)_{n \geq 1} \subset \R_+$, $(k_n)_{n \geq 1} \subset \Z^2$ such that
\begin{equation}
\label{eq: k_condition}
y_n \rightarrow 0, \hskip1cm v(y_n) \in \langle k_n \rangle.
\end{equation}
Define
\begin{equation}
\label{defF}
f(\theta, R) = \sum_{n\geq 1} \epsilon _n a_n(R)\cos(\theta \cdot k_n^{\bot}), \end{equation}
for some constants $\epsilon_n \in \R$ and some functions $a_n: \R^2 \rightarrow \R$ of disjoint support, flat at $0$, such that 
\begin{equation}
\label{aconstant}
a_n\mid_{\Lambda_{y_n}} = 1.
\end{equation} 
If $f$ is at least of class $C^1$ then for any $\epsilon > 0$ the Hamiltonian vector field associated to $H_\epsilon$ is given by
\begin{equation*}
\label{Hvectfield}
 X_{H_\epsilon}(\theta, R) = \Big(\nabla h(R) + \epsilon\cos(\theta \cdot k^\bot)a_n(R), \sum_{n\geq 1} \epsilon\epsilon_n a_n(R) \sin(\theta \cdot k_n^{\bot})k_n^{\bot}\Big).
\end{equation*}
Hence for any $\theta^{(n)} \in \T^2$ such that $\sin(\theta^{(n)} \cdot k_n^{\bot}) = 1$ and any $R^{(n)} \in \Lambda_{y_n}$ the previous equation becomes
 \begin{equation*}
\label{flowalongline}
 X_{H_\epsilon}\big(\theta^{(n)}, R^{(n)}\big) = \big(\nabla h(R^{(n)}), \epsilon\epsilon_nk_n^{\bot}\big). 
\end{equation*} 
Then, by the same arguments of last section, the solution $z^{(n)}(t)$ with initial conditions 
$\left(\theta^{(n)},R^{(n)}\right)$ is completely contained in $\T^2 \times \Lambda_{y_n}$ and obeys 
 \begin{equation} 
 \label{linearspeed}
 z_R^{(n)}(t) = R^{(n)} + t\epsilon\epsilon_n k_n^{\bot},
 \end{equation}
where we denote $z^{(n)}(t) = \big(z_\theta^{(n)}(t), z_R^{(n)}(t)\big)$. Therefore, the functions $h,$ $f$ defined above will satisfy the conclusions of Theorem \ref{thm: actionangle}. We formalize this construction in the following proposition. 

\begin{prop}
\label{prop: sequences}
Fix $\sigma > 0$. Let $v: J = (-1,1) \rightarrow \R^2$ analytic obeying: 
\begin{equation}
\label{condv}
v_2(t) \neq 0, \hskip1cm v_1'(t)v_2(t) \neq v_1(t)v_2'(t).
\end{equation}
Then there exist $h \in \mathscr{C}^\omega(0),$ $(y_n)_{n \in \N} \subset \R_+,$ $(k_n)_{n \in \N} \subset \Z^2$, $(a_n)_{n \in \N} \subset \mathscr{C}^\infty(0)$ verifying (\ref{gradyaxis}), (\ref{eq: k_condition}), (\ref{aconstant}) and positive constants $(\epsilon_n)_{n \in \N}$ such that $f$ as defined in (\ref{defF}) belongs to $\mathcal{G}^{1+\sigma}(\T_0)$ and is flat at $\T_0$. If in addition $v$ satisfies one of the following conditions: 
\begin{gather}
\label{Kol1}
\begin{array}{ccc}
v_1(0) = 0, & & v_1'(0) \neq 0, 
\end{array} \\
\nonumber \text{or} \\
\label{Kol2}
\begin{array}{ccc}
v_1(0)v_2'(0) \neq 0, & & v_1'(0)=0,
\end{array}
\end{gather}
then $h$ will be Kolmogorov non-degenerate. 
\end{prop} 

Before proving Proposition \ref{prop: sequences} let us show how it implies Theorem \ref{thm: actionangle}. 

\begin{proof}[Proof of Theorem \ref{thm: actionangle}.] Suppose $\omega_2 \neq 0$. Take any analytic path $v: (-1, 1) \rightarrow \R$ with $v(0) = \omega$ satisfying (\ref{condv}) and obeying (\ref{Kol1}) or (\ref{Kol2}). Let $h$, $f$, $(y_n)_{n \in \N}$, $(k_n)_{n \in \N},$ $(a_n)_{n \in \N}$ as in Proposition \ref{prop: sequences} when applied to $v$. Notice that in this case $h$ is Kolmogorov non-degenerate, $f$ is flat at $\T_0$ and for any $\epsilon \in \R$ the torus $\T_0$ is invariant for the Hamiltonian $H_\epsilon = h + \epsilon f$ and has rotation vector $\omega.$

For all $n \in \N$ let $\theta^{(n)} \in \T^2$ such that $\sin(\theta^{(n)} \cdot k_n^{\bot}) = 1$. By construction, for $\epsilon \in \R$ fixed, the solution $z^{(n)}(t) = \big(z_\theta^{(n)}(t), z_R^{(n)}(t)\big)$ of the Hamiltonian flow of $H_\epsilon$ with initial conditions $(\theta^{(n)},(0,y_n)) $ obeys (\ref{linearspeed}) with $R^{(n)} = (0, y_n),$ namely
\begin{equation}
\label{eq: proof_diffusion}
z_R^{(n)}(t) = (0,y_n) + t\epsilon \epsilon_n k_n^{\bot},
\end{equation}
for every $t \in \R$ such that the flow is defined. This clearly implies the instability of the invariant torus $\T_0$. The argument for $\omega_1 \neq 0$ is completely analogous.
\end{proof}

\subsection{Proof of Proposition \ref{prop: sequences}}

We will split the construction in several steps. Lemma \ref{lem: intham} concerns the integrable Hamiltonian $h$. Explicit definitions of $(y_n)_{n \in \N},$ $(k_n)_{n \in \N}$ will be made in Lemma \ref{lemmaseq} and the functions $(a_n)_{n \in \N}$ will be constructed in Lemma \ref{boundsa}. 
\begin{lem}
\label{lem: intham}
Let $v: J \rightarrow \R^2$ as in Proposition \ref{prop: sequences} and denote $\Lambda_t = (0, t) + \langle v^\perp(t) \rangle$. Define $\varphi(t) = \frac{v_1(t)}{v_2(t)}$ and suppose that $0< |\varphi ' (t)| < \beta$ for every $t \in J$. Then there exist $\delta > 0$, an open set $U \subset \R^2$ and an analytic function $h: U \rightarrow \R$ such that:
 
 \begin{enumerate}
 \item $U = \bigsqcup_{t\in J} \Lambda_t^\delta$ where $\Lambda_t^\delta = \Lambda_t \cap (-\delta,\delta) \times \R$. In particular $ \lbrace 0 \rbrace \times J \subset U$.
 \item $h$ obeys (\ref{gradyaxis}) for every $t \in J$.
 \item If $v$ satisfies (\ref{Kol1}) or (\ref{Kol2}) then $h$ is Kolmogorov non-degenerate.
 \end{enumerate}
\end{lem}

\begin{proof} Let $\delta = 1/\beta$ and define $\phi: (-\delta,\delta) \times J \rightarrow \R^2$ by 
\[ \phi(x,y) = \left( x,y-x \varphi(y) \right).\]
Denote $V= (-\delta,\delta) \times J$ and $U = \phi(V)$. We will show that $\phi$ is an analytic diffeomorphism. Given $(x,y) \in V$ 
\[ \det (D\phi) (x,y) = 1-x\varphi '(y)\]
which is non-zero by the definition of $\delta$. Thus $\phi $ is a local analytic diffeomorphism. To prove injectivity let $(x,y),$ $(x',y') \in V$ be such that $\phi (x,y) = \phi(x',y')$. This clearly implies $x = x'$ and 
\[ y - x\varphi(y) = y'-x\varphi(y').\]
If $y \neq y'$ we would have \[0 < |y - y'| = |x||\varphi(y)-\varphi(y')| < \delta \beta | y-y'| = |y-y'| \] which is impossible. Thus $y = y'$. This shows that $\phi$ is injective and therefore an analytic diffeomorphism. Notice that $\phi((-\delta,\delta) \times \lbrace t \rbrace ) = \Lambda_t^\delta$ for every $t \in J$ and thus $U$ verifies the first assertion in the lemma. Define $g: J \rightarrow \R$, $h: U \rightarrow \R$ by 
\[ g(y)= \int_0^y v_2(t)dt, \hskip1cm h(x,y) = g \circ \pi _2 \circ \phi^{-1}(x,y), \]
where $\pi_2 : \R^2 \rightarrow \R$ denotes the projection to the second coordinate. For all $y \in J$, $w \in \Lambda_y^\delta$, we have $\pi_2 \circ \phi^{-1} (w) = y$ which implies $h(w) = g(y).$ Therefore $h \mid_{\Lambda_y^\delta}$ is constant and $\langle \nabla h(w), v^{\bot}(y) \rangle = 0$ for every $ y \in J$ and every $w \in \Lambda_y^\delta$. In particular 
 \begin{equation} 
 \label{gradh} 
\partial_x h(0,y) v_2(y) - \partial_y h(0,y)v_1(y) = 0.
 \end{equation} 
As $\phi^{-1}(0,y) = (0,y)$ it follows that $\partial_y h (0,y) =g'(y) = v_2(y).$  By (\ref{gradh}), for any $y \in J$ $\partial _x h (0,y) = v_1(y)$ and thus 
\[\nabla h (0,y) = v(y),\]
 which proves the second assertion. It remains to show $h$ is Kolmogorov non-degenerate whenever (\ref{Kol1}) or (\ref{Kol2}) are satisfied. Differentiating the equation above
\[ \partial_y^2 h (0) = v_2'(0), \hskip1cm \partial_{xy}^2 h(0)= v_1'(0). \]
A simple calculation leads to 
\[ \partial_{xx}^2h(0) = v_2'(0)\varphi^2(0) + 2v_2(0)\varphi(0)\varphi'(0). \]
Hence 
\[ \det|\partial^2h (0)| = (v_2'(0)\varphi(0))^2 + 2v_2(0)v_2'(0)\varphi(0)\varphi'(0) - (v_1'(0))^2. \] 
If $v_1(0) =0$, $v_1'(0) \neq 0$
 \[ \det|\partial^2h (0)| = -(v_1'(0))^2 \neq 0 .\] 
If $v_1(0)v_2'(0) \neq 0$, $v_1'(0)=0$
\[\varphi'(0) = -\varphi(0)\dfrac{v_2'(0)}{v_2(0) }, \]
and we have
 \[ \det|\partial^2h (0)| = - (v_2'(0)\varphi(0))^2 \neq 0. \]
 \end{proof} 
\begin{lem}
\label{lemmaseq}
Let $v: J \rightarrow \R^2$ and $\varphi: J \rightarrow \R$ as in Lemma \ref{lem: intham}. There exist sequences $(y_n)_{n \geq 1} \subset J,$ $(k_n)_{n \geq 1} \subset \Z^2$ and a constant $C > 0$ such that: 
\begin{enumerate}
\item $0 < 2y_{n+1} \leq y_n,$
\item $v(y_n) \in \langle k_n \rangle$ \text{ and } $\vert k_n \vert_{\max} \leq C/y_n$. \end{enumerate}
\end{lem}

\begin{proof} Let $\psi : J \rightarrow \R$ given by $\psi(t)= \tan^{-1}(\varphi(t)).$ This function measures the angle of $v(t)$ with the $x$-axis. From the hypotheses there exists $\rho$ positive, such that $\psi '(0)= 2\rho > 0.$ Let $t_0 > 0$ such that $\psi '(t) > \rho $ for $|t| < t_0$. Then for every interval $I \subset (-t_0,t_0)$ we have $|\psi(I)| > \rho|I|,$ where we denote by $| \cdot |$ the Lebesgue measure of the interval. Let 
\[ J_n = \left[ \dfrac{1}{2^{2n}\rho}, \dfrac{1}{2^{2n-1}\rho} \right]. \] 
For $n$ sufficiently large $J_n$ is contained in $(-t_0,t_0)$ and thus $|\psi(J_n)| > 1/2^{2n}.$

\begin{claim}
Let $I \subsetneq \mathbb{T}$ be an interval such that $|I| > \frac{1}{n}.$ There exists $k \in \Z^2$ obeying \[ |k|_{\max} = n, \hskip1cm \dfrac{k}{\Vert k \Vert _2} \in I.\]
\end{claim}
\begin{proof}[Proof of the Claim.] Let \[Z_n = \left\lbrace \dfrac{k}{\Vert k \Vert _2} \bigg\rvert |k|_{\max} = n \right\rbrace.\]
This set splits $\mathbb{T}$ in $8n$ intervals. It suffices to note that the biggest interval has size \[ \sin^{-1}\left( \dfrac{1}{\sqrt{n^2+1}}\right) < \dfrac{1}{n} .\]
\end{proof} 
By the claim there exists $k_n \in \Z^2$ such that \[|k_n|_{\max} = 2^{2n}, \hspace{1cm} \dfrac{k}{\Vert k \Vert _2} \in \psi(J_n).\]
Since $\psi$ measures the angle of $v$ with the $x$-axis, there exists $y_n \in J_n$ such that $v(y_n)$ belongs to $\langle k_n \rangle $. Defining $y_n$, $k_n$ in this fashion we obtain the result. 
\end{proof}

\begin{lem}
\label{boundsa}
Let $v: J \rightarrow \R^2$ as in Lemma \ref{lem: intham} and $(y_n)_{n \in \N}$, $(k_n)_{n \in \N}$ as in Lemma \ref{lemmaseq}. Denote $\Lambda_t = (0, t) + \langle v^\perp(t) \rangle$. For every $\gamma > 0$ there exist a sequence of functions $\lbrace a_n : U \rightarrow \R \rbrace_{n\geq 1}$ in $G^{1+\gamma}(U)$, obeying (\ref{aconstant}), such that:

\begin{enumerate}
\item $\textup{Supp}(a_n) \subset \lbrace R \in \R^2 \mid d(R,\Lambda_{y_n}) < y_n/4 \rbrace$.
\item $\textup{Supp}(a_n) \cap \textup{Supp}(a_m) = \emptyset$ for $n \neq m$ sufficiently large.
\item There exist constants $c_a,\rho_a > 0$ such that \[ \Vert \partial^\alpha a_n \Vert _U \leq c_a\left(\dfrac{\rho_a}{y_n}\right)^{|\alpha|}(\alpha !)^{1+\gamma} \]
for any multi-index $\alpha$ and any $n \geq 1$. 
\end{enumerate}
\end{lem}

\begin{proof}
Let $a: \R \rightarrow \R$
\[ a(t-1/2) = \left\lbrace \begin{array}{ccc} 0 & \text{ if } & t \notin (0,1), \\
\exp\left(\dfrac{-1}{((1-t)t)^{\gamma /2}}\right) &\text{ if } & t \in (0,1).
\end{array} \right. \]
Denote $\delta_n = y_n/4,$ $w_n = \big(w_1^{(n)},w_2^{(n)}\big) = \frac{v^{\bot}(y_n)}{\Vert v^{\bot}(y_n)\Vert_2}$ and let $a_n : U \rightarrow \R$
\[ a_n(R) = a\left( \dfrac{\langle R, v_n \rangle - y_n w_2^{(n)}}{\delta_n} \right). \]
Then $a_n$ is a smooth function obeying (\ref{aconstant}) and such that $a_n(R)=0$ whenever $d(R,\Lambda_{y_n}) > y_n/4$. This proves the first assertion. The second assertion follows easily from the first and the definition of $y_n$. The function $a$ belongs to $G^{1+\gamma /2}(\R)$ (see \cite{ramis_devissage_1978}), thus there exist constants $m,$ $c > 0$ such that \[ \Vert a^{(k)} \Vert_R \leq m c^k(k!)^{1+\gamma /2} \]
Given a multi-index $\alpha = (\alpha_1,\alpha_2) \in \N^2$ 
\[ \partial^{\alpha} a_n (R) = \frac{{w_1^{(n)}}^{\alpha_1}{w_2^{(n)}}^{\alpha_2}}{\delta_n^{|\alpha|}}a^{(|\alpha|)}\left( \dfrac{\langle R, v_n \rangle - y_n w_2^{(n)}}{\delta_n} \right). \]
Thus
 \[ \Vert \partial^\alpha a_n \Vert _U \leq m\left(\dfrac{c}{\delta_n}\right)^{|\alpha|}(|\alpha|!)^{1+\gamma /2}. \] 
Note that $(|\alpha|!)^{1+ \gamma /2} \leq (\alpha !)^{1+\gamma}$ for $|\alpha|$ sufficiently large. So there exists $m' >0$ such that 
 \[ \Vert \partial^\alpha a_n \Vert _U \leq mm'\left(\dfrac{4c}{y_n}\right)^{|\alpha|}(\alpha!)^{1+\gamma}.\]
 Take $c_a = mm'$, $\rho_a = 4c$.
\end{proof}
 
\begin{proof}[Proof of Proposition \ref{prop: sequences}.] Let $h: U \rightarrow \R$ as in Lemma \ref{lem: intham}. By construction, $h$ depends only on the analytic function $v$ and it will be Kolmogorov non-degenerate if conditions (\ref{Kol1}) or (\ref{Kol2}) are satisfied. Let $h$ as in Lemma \ref{lem: intham}, $(y_n)_{n \in \N}$, $(k_n)_{n \in \N}$ as in Lemma \ref{lemmaseq} and $(a_n)_{n \in \N}$ as in Lemma \ref{boundsa} with $\gamma = \sigma / 2$. Denote $\rho = 1/\sigma$ and define $\epsilon_n = \exp(-1/y_n^\rho)$ for all $n \in \N$. 

It follows from the previous definitions that $f$, which is given by (\ref{defF}), is a well defined function of class $C^\infty$ on $\mathbb{T}^2 \times (U \,\setminus\, \lbrace 0 \rbrace)$. Thus it suffices to show that $f$ admits continuous derivatives of all orders on $\T^2 \times \lbrace 0 \rbrace$ and that it is a Gevrey-smooth function. We may suppose WLOG that condition 2 of Lemma \ref{boundsa} is true for all $n \neq m$. Indeed if this were not the case it suffices to take $\delta$ in Lemma \ref{lem: intham} smaller as to guarantee 
\[ (-\delta,\delta) \times \R \cap \lbrace R \in \R^2\mid d(R, \Lambda_{y_n}) < y_n/4, d(R, \Lambda_{y_m}) < y_m/4 \rbrace= \emptyset.\] 
Let $\alpha = (\tau,\beta) \in \N^2 \times \N^2$ and $(\theta,R) \in \textup{Supp}(a_n)$. By Lemmas \ref{lemmaseq} and \ref{boundsa} 
\begin{align*}
|\partial^\alpha f(\theta, R)| & \leq |\epsilon _n (\partial ^{\tau} a_n)(R)||k_n^{\bot}|_{\max}^{|\beta|} \\
 & \leq \epsilon_n c_a\left(\dfrac{\rho_a}{y_n}\right)^{|\tau|}(\alpha !)^{1+\gamma} \left( \dfrac{C_1}{y_n} \right)^{|\beta|} \\
& \xrightarrow[n \rightarrow \infty]{} 0 .
\end{align*} 
Thus $f$ is infinitely derivable on $\T^2 \times \lbrace 0 \rbrace$. Moreover $f$ is flat on this set. By Lemma \ref{estimate} (which is proved below) there exists $C_2 > 0$, independent of $\alpha$ and $n$, such that \[ \dfrac{\epsilon_n}{y_n^{|\alpha|}} = \dfrac{e^{-1/y_n^{\rho}}}{y_n^{|\alpha|}} \leq C_2(\alpha!)^{\gamma}.\] 
Hence taking $c =c_a C_2$ and $\rho = \max\lbrace \rho_a,C_1\rbrace$ 
\begin{align*}
\Vert \partial^\alpha (H-h) \Vert_U & \leq \sup_{n \geq 1} \epsilon _n \Vert (\partial ^{\alpha} a_n)\Vert_U|k_n^{\bot}|_{\max}^{|\beta|} \\
& \leq c_aC_2 \rho_a^{|\tau|} C_1^{|\beta|} (\alpha !)^{1+2\gamma}\\ 
& \leq c\rho^{|\alpha|} (\alpha !)^{1+\sigma}.
\end{align*} 
Therefore $f \in G^{1+\sigma}(0)$.
\end{proof}

\begin{lem}
\label{estimate}
Let $d \in \N$, $\gamma >0$. There exist a constant $C >0$ , such that 
\[ \dfrac{\exp\left(-y^{-2/\gamma}\right)}{y^{|\alpha|}} \leq C (\alpha!)^{\gamma}\] for every $y > 0$ and every $\alpha \in \N^d$. 

\end{lem}
\begin{proof} Fix $k \in \N$ and let $f_k: \R_+ \rightarrow \R$ \[ f_k(y) = \dfrac{1}{y^k}\exp\left(\dfrac{-y^{2/\gamma}}{d}\right).\] 
This function clearly converges to 0 as $y \rightarrow 0$. Calculating its derivative we can conclude that it attains it maximum at the point \[ y = \left( \dfrac{2}{d\gamma k}\right)^{\gamma/2}.\] 
Using Stirling's formula 
\begin{align*} 
\lim_{k \rightarrow \infty} \dfrac{\max\limits_{t \in \R} f_k(t)}{k!^\gamma} & \leq \lim_{ k \rightarrow \infty} \dfrac{ (de \gamma k)^{\gamma k /2}}{ 2^{\gamma k/2}k^{\gamma k}(2\pi k)^{\gamma/2}} \\
& \leq \lim_{ k \rightarrow \infty} \dfrac{(de\gamma )^{\gamma k/2}}{\gamma^{\gamma/2}k^{k\gamma/2}} \\
& = 0 .
\end{align*}
Thus there exists a constant $M$ depending only on $d,\gamma$ such that \[ \dfrac{1}{y^k}\exp\left(\dfrac{-y^{2/\gamma}}{d}\right) \leq M k!^\gamma\] for every $k \geq 1$ and every $y \in \R_+$. Hence, for $C = M^d$ and  for every $\alpha \in \N^d$
\[ \dfrac{\exp\left(-y^{-2/\gamma}\right)}{y^{|\alpha|}} = \prod_{i=1}^{d} \dfrac{1}{y^{|\alpha_i|}}\exp\left(\dfrac{-y^{-2/\gamma}}{d}\right) \leq C (\alpha!)^{\gamma}.\]
\end{proof} 

 \section{Instability for elliptic equilibria}
 \label{sec: fixedpoint_instability}
 
As mentioned at the beginning of this work, Hamiltonians in action-angle coordinates and Hamiltonians near an elliptic fixed point are not equivalent settings. Nevertheless, we can relate them by means of the symplectic change of coordinates 
\begin{equation}
\label{eq: symplectic_polar}
\begin{array}{cccc}
 T: & \T^2 \times \R^2_+ & \rightarrow & \R^2_\ast \times \R^2_\ast \\
 & (\theta,R) & \mapsto & \left(\sqrt{2R_1}e^{i\theta_1},\sqrt{2R_2}e^{i\theta_2}\right)
 \end{array},
 \end{equation} 
where \[ \R^2_+ = \lbrace R \in \R^2 \mid R_1,R_2 > 0\rbrace, \hspace{2cm} \R^2_* = \R^2\, \setminus \,\{ 0 \}.\]
Recall that since $T$ is symplectic, that is 
$T^*\left(\sum_{i = 1}^2d\theta_i\wedge dR_i\right) = \sum_{i = 1}^2dx_i \wedge dy_i,$
for any Hamiltonian $H$ defined on an open subset of $\T^2 \times \R^2_+$ we have 
\[ \Phi^t_{H \circ T^{-1}}(p) = T \circ \Phi^t_{H} \circ T^{-1}(p),\]
for any $p \in \R^2_\ast \times \R^2_\ast $ and any $t \in \R$ such that the Hamiltonian flows are well defined. In the following we denote by $I$ the function $I = (I_1,I_2): \R^4 \rightarrow \R^2$ given by 
 \[ I_i (x, y) = \dfrac{x_i^2+y_i^2}{2}.\]
By adapting the examples constructed in Theorem \ref{thm: actionangle}, and with the help of the symplectic change of coordinates $T$, we will prove the following.

\begin{thm}
\label{thm: fixedpoint}
Given $\sigma > 0$ and $\omega \in \R^2 \,\setminus\,\lbrace 0 \rbrace$ such that $\omega_1 \omega_2 < 0$ there exist $\overline{h} \in \mathscr{C}^\omega(0)$ and $\overline{f} \in \mathcal{G}^{1+\sigma}(0)$ flat at $0$ such that for any $\epsilon > 0$ the origin is a Lyapunov unstable Kolmogorov non-degenerate elliptic fixed point with frequency vector $\omega$ for the Hamiltonian
 \[\overline{H}_\epsilon(x, y) = \overline{h}(I(x, y)) + \epsilon \overline{f}(x, y).\]
\end{thm} 

Notice that, contrary to the action-angle case where we were able to construct Lyapunov unstable invariant tori with arbitrary rotation vector, an elliptic fixed point whose rotation vector $\omega$ is non-resonant and verifies $\omega_1 \omega_2 > 0$ must be Lyapunov stable. In fact, as discussed in Section \ref{UE: sc_setting}, we can suppose WLOG that the Hamiltonian associated to such an elliptic point is of the form
\[ H(x, y) = \sum_{i = 1}^2 \omega_i \frac{x_i^2 + y_i^2}{2} + \mathcal{O}^3(x, y).\]
If $\omega_1 \omega_2 > 0$ the energy surfaces $H^{-1}(e)$ must be compact for $e$ sufficiently small. Since the orbits of the Hamiltonian flow are always contained in the energy surfaces of the Hamiltonian the elliptic fixed point must be Lyapunov stable. 

Let us sketch the proof of Theorem \ref{thm: fixedpoint}. Suppose $f,$ $h$ as in Theorem \ref{thm: actionangle}. Define
 \[ \overline{f} = f \circ T^{-1}, \hskip1cm \overline{h} = h \circ T^{-1},\]
with $T$ as in (\ref{eq: symplectic_polar}). Notice that $\overline{f}, \overline{h}$ are only defined on $V \cap (\R^2_* \times \R^2_*)$, for some open neighbourhood $V$ of the origin. Since $h$ depends only on $R$ it is easy to show that $\overline{h}$ extends to an open neighbourhood of $0$. Moreover, the origin is an elliptic fixed point with frequency vector $\omega$ for the integrable Hamiltonian $\overline{h}$. We would like to extend $\overline{f}$ to an open neighbourhood of $0$ so that the origin is an elliptic fixed point of $\overline{H}_\epsilon$ and to use the unstable orbits associated to the Hamiltonian flow of $H_\epsilon = h + \epsilon f$ to prove Theorem \ref{thm: fixedpoint}. This approach poses two main difficulties: 
\begin{itemize}
\item The unstable orbits of $H_\epsilon$ might not be contained in $T^{-1}(\R^2_* \times \R^2_*).$
\item A smooth extension of $\overline{f}$ to $V$ may not exist. 
\end{itemize}
Recall that the unstable orbits $z^{(n)}(t) = \big(z_\theta^{(n)}(t), z_R^{(n)}(t)\big)$ associated to $H_\epsilon$, which were described in detail Section \ref{sketch}, are contained in sets of the form $\T^2 \times \Lambda_{y_n}$, where $\Lambda_{y_n} = (0, y_n) + \langle k_n \rangle$ for some $y_n \in \R,$ $k_n \in \Z^2$. Moreover, they obey (\ref{eq: proof_diffusion}), namely
\[ z_R^{(n)}(t) = (0, y_n) + t\epsilon\epsilon_n k_n^{\bot},\]
for all $t \in \R$ such that the flow is defined. Therefore the first difficulty can be solved if we show that $\Lambda_{y_n} \cap \R^2_+$ contains an infinite segment of $\Lambda_{y_n}$. By looking at the definition of the vectors $k_n$ (Lemma \ref{lemmaseq}), this will be true if and only if the frequency vector $\omega$ satisfies $\omega _1 \omega_2 < 0.$ 

To overcome the second difficulty it suffices to note that an extension of $\overline{f}$ to an open neighbourhood of the origin exists if $f$ is flat on the boundary of $\R^2_+$ that we denote by $\partial\R^2_+$. Notice that in this case the origin will be an elliptic equilibrium with rotation vector $\omega$ for $\overline{H}_\epsilon$. Let us recall the definition of $f$ which is given by (\ref{defF}), namely
\[ f(\theta, R) = \sum_{n\geq 1} \epsilon _n a_n(R)\cos(\theta \cdot k_n^{\bot}),\]
where $\epsilon_n \in \R,$ $k_n \in \Z^2$ and $a_n \in \mathcal{C}^\infty(0).$ To guarantee that $f$ is flat on $\partial\R^2_+$ we will modify the sequence of functions $a_n$, constructed in Lemma \ref{boundsa}, so that its support will be contained in $\R^2_+$. We do this by considering a new family of bump functions $b_n$ (Lemma \ref{boundsb}) and taking the product $a_nb_n$. Following these ideas, instead of taking $f$ as in (\ref{defF}), we will consider a function $F$ of the form
 \begin{equation} 
\label{deff1}
F(\theta, R) = \sum_{n\geq 1} \delta_n \epsilon _n a_n(R)b_n(R_1)\cos(\theta \cdot k_n^{\bot})
\end{equation}
for some constants $\delta_n$ converging to $0$. Finally, we will show that for appropriate sequences $b_n$, $\delta_n$ the function $\overline{f} = F \circ T^{-1}$ admits a Gevrey-smooth extension to an open neighbourhood of $0$ and that for any $\epsilon > 0$ the origin will be a Lyapunov unstable elliptic point for the Hamiltonian $\overline{H}_\epsilon $

\subsection{Proof of Theorem \ref{thm: fixedpoint}}

We start by defining the bump functions $b_n$.

\begin{lem}
\label{boundsb}
Let $\gamma > 0$. Given a sequence of positive numbers $(y_n)_{n\in \N}$ there exists a sequence of functions $\lbrace b_n :\R \rightarrow \R \rbrace_{n\geq 1}$ in $G^{1+\gamma}(\R)$ such that for every $n \geq 1$

\begin{enumerate}
\item $b_n(t) = 0$ for $t \leq y_n$.
\item $b_n(t) = 1$ for $t \geq 2y_n$.
\item There exist constants $c_b,\rho_b > 0$ such that \[ \Vert \partial^\alpha b_n \Vert _\R \leq c_b\left(\dfrac{\rho_b}{y_n}\right)^{|\alpha|}(\alpha !)^{1+\gamma} \]
for every $\alpha \in \N$. 
\end{enumerate}
\end{lem}
\begin{proof} Let $f,$ $b: \R \rightarrow \R$ 
\[ f(t) = \left\lbrace \begin{array}{ccc} 0 & \text{ if } & t < 0 \\
\exp\left( -\dfrac{1}{t^\gamma} \right) &\text{ if } & t \geq 0
\end{array} \right. ,\]
\[ b(t) = \dfrac{f(t)}{f(t)+f(1-t)}.\]
We have $f \in G^{1+\gamma}$. A proof of this can be found in the appendix of \cite{marco_stability_2002}. As Gevrey classes are closed under addition, products and reciprocals (as long as the function does not get arbitrarily close from 0) it is clear that $b \in G^{1+\gamma}$. Thus there exist constants $c_b, \rho_b > 0 $ such that 
\[ \Vert \partial^\alpha b \Vert _\R \leq c_b \rho_b^{|\alpha|}(\alpha !)^{1+\gamma} \]
for every $\alpha \in \N$. Define $b_n: \R \rightarrow \R$
\[ b_n(t) = b\left( \dfrac{x-y_n}{y_n} \right).\]
\end{proof}

To analyse $\overline{f}$ it will be useful to study the composition of the functions $\cos(\theta_i)$ and $\sin(\theta_i)$ with $T^{-1}$. Define $g : \R^2_* \times \R^2_* \rightarrow \R^4$ as
\[g = (\cos(\theta_1), \sin(\theta_1), \cos(\theta_2), \sin(\theta_2)) \circ T^{-1}.\] 
It is easy to show that
 \[ g(x, y) = \left( \dfrac{x_1}{(x_1^2+y_1^2)^{1/2}},\dfrac{y_1}{(x_1^2+y_1^2)^{1/2}} ,\dfrac{x_2}{(x_2^2+y_2^2)^{1/2}} ,\dfrac{y_2}{(x_2^2+y_2^2)^{1/2}} \right). \] 
In the following we denote 
\begin{equation}
\label{eq: S_t}
S_t = \left\{ (x_1,y_1,x_2,y_2) \in \R^2_* \times \R^2_* \mid \max(x_1^2+y_1^2, x_2^2+y_2^2) \geq t^2 \right\}.
\end{equation}

 \begin{lem}
 \label{boundsg}
 There exist constants $c_g, \rho_g > 0$ such that 
 \[\sup_{z \in S_t} |\partial ^ \alpha g (z)| \leq c_g\left(\dfrac{\rho_g}{t}\right)^{|\alpha|}\alpha!\]
 for every $t > 0$ and every $\alpha \in \mathbb {N}^4$.
 \end{lem} 
 \begin{proof}
 Denote $g = (g_1, g_2)$. We will prove the bounds for the function $g_1$ as those for $g_2$ are analogous. Since $g_1$ only depends on two variables we can consider it as being defined on $\R^2$. Notice that for every $\lambda > 0$, and every $(x,y) \in \R^2$
 \[ g_1(\lambda x \lambda y ) = g_1(x,y).\] 
Thus 
 \[ \partial^\alpha g_1 (\lambda x, \lambda y ) = \dfrac{1}{\lambda^{|\alpha|}}\partial^\alpha g_1 (x,y)\] 
 for every $\alpha \in \N^2$. Since $g_1$ is analytic on a small open neighbourhood of $\mathbb{S}^1 \subset \R^2$ there exist constants $c_{g_1}, \rho_{g_1}$ such that 
 \[ \sup_{(x,y) \in \mathbb{S}^1} |\partial ^ \alpha g_1 (x,y)| \leq c_{g_1}\rho_{g_1}^{|\alpha|}\alpha!. \] 
Given $t > 0$, $\alpha \in \N^2$
 \begin{align*}
 \sup_{(x,y) \in S_t} |\partial ^ \alpha g_1 (x,y)| & = 
 \sup_{t \leq \lambda } \sup_{(x,y)\in \mathbb{S}^1} |\partial ^ \alpha g_1 (\lambda x,\lambda y)| \\ 
 & \leq \dfrac{1}{t^{|\alpha|}} \sup_{(x,y) \in \mathbb{S}^1} |\partial ^ \alpha g_1 (x,y)| \\
 & \leq c_{g_1}\left(\dfrac{\rho_{g_1}}{t}\right)^{|\alpha|}\alpha!.
 \end{align*}
The bounds for the other functions are completely analogous. Taking maximum over the constants the result follows.
\end{proof}

\begin{proof}[Proof of Theorem \ref{thm: fixedpoint}.] Suppose $\omega_2 \neq 0$. As in the proof of Theorem \ref{thm: actionangle} take any analytic path $v: (-1, 1) \rightarrow \R^2$ with $v(0) = \omega$ satisfying (\ref{condv}) and obeying (\ref{Kol1}) or (\ref{Kol2}). Let $h$, $f$, $(y_n)_{n \in \N}$, $(k_n)_{n \in \N},$ $(a_n)_{n \in \N}$ as in Proposition \ref{prop: sequences} when applied to $v$ and $(b_n)_{n \in \N}$ as in Lemma \ref{boundsb} when applied to the sequence $(y_n)_{n \in \N}$ with $\gamma = \frac{\sigma}{2}.$ Define $F \in \mathscr{C}^\infty(\T_0)$ as in (\ref{deff1}), namely
\[ F(\theta, R) = \sum_{n\geq 1} \delta_n \epsilon _n a_n(R)b_n(R_1)\cos(\theta \cdot k_n^{\bot}),\]
 for some positive constants $(\delta_n)_{n \in \N}$ that we will specify later. We will show that 
\[ \overline{h} = h \circ T^{-1}, \hskip1cm \overline{f} = F \circ T^{-1},\]
 where $T$ is the symplectic change of coordinates given by (\ref{eq: symplectic_polar}), can be extended to an open neighbourhood of the origin and that they verify the conclusions of Theorem \ref{thm: fixedpoint}. Since by construction $h$ depends only on $R$ and 
\[ \pi_R \left( T^{-1}(x, y)\right) = I(x, y),\]
where $\pi_R : \T^2 \times \R^2 \rightarrow \R^2$ denotes the canonical projection, it is clear that $\overline{h}$ admits an extension to an open neighbourhood of $0$. As an abuse of notation we will denote this extension also by $\overline{h}.$ Let 
\[ f_n(\theta,R) = a_n(R)b_n(R_1)\cos(\theta \cdot k_n^\bot) \]
and define $\overline{f_n} = f_n \circ T^{-1}.$ Recall that by definition of $a_n$ and $b_n$ 
\[\textup{Supp}(a_n) \subset \lbrace R \in \R^2 \mid d(R,\Lambda_{y_n}) \leq y_n/4 \rbrace, \hskip1cm \textup{Supp}(b_n) \subset \lbrace R \in \R^2 \mid R_1 \geq y_n \rbrace,\]
where $\Lambda_t = (0, t) + \langle v^\perp(t) \rangle.$ Since $v_1(0) v_2(0) = \omega_1 \omega_2 < 0,$ we have $v_1(t) v_2(t) < 0$ for $t$ sufficiently small which implies 
\[ \Lambda_{t} \cap ((t ,\infty)\times \R) \subset (t,\infty) \times (t,\infty).\]
Thus for $n$ sufficiently large
\begin{align*}
\textup{Supp}(\overline{f_n}) 
& \subset S_{3y_n/4},
\end{align*} 
with $S_t$ as defined in (\ref{eq: S_t}). We may assume WLOG that the previous relation holds all $n \in \N$. In fact, if this were not the case, we can set $a_n$ equal to zero for all $n \leq N$ for some $N$ sufficiently large. Hence $\overline{f_n} = f_n \circ T^{-1}$ can be extended by zero to an open neighbourhood $V$ of the origin in a smooth way. Notice that this extension will be flat on $V \cap \partial(\R^2_*\times \R^2_*)$. Thus the function $\overline{f}$, which is equal to the sum over $n$ of the functions $\overline{f_n}$, can be extended by zero to a continuous function on $V$ . As an abuse of notation we will denote this extensions also by $\overline{f_n}$ and $\overline{f}$. 

Let us show that $\overline{f}$ is of class $C^\infty$ on $V.$ By definition $\overline{f}$ is of class $C^\infty$ on $V \cap \R^2_*\times \R^2_*$ and since $\overline{f}$ it is equal to 0 outside this set it suffices to show that $\overline{f}$ is infinitely differentiable on $\partial(\R^2_*\times \R^2_*)$. Notice that $\cos(\theta \cdot k_n)$ can be expressed as \[ p_n(\cos(\theta_1), \sin (\theta_1),\cos(\theta_2), \sin(\theta_2)),\]
for some polynomial $p_n$ of degree at most $|k_n|$. Thus 
\[ \overline{f_n} = (b_n \circ I_1)(a_n \circ I)(p_n \circ g). \]
Let $c_n = \Vert p_n \Vert_{C^{|k_n|}(U)} $, $d = \Vert I \Vert_{C^{2}(V)} $. Since $\textup{Supp}\left(\overline{f_n}\right) \subset S_{3y_n/4} $ the bounds for $a_n$, $b_n$, $g$ in Propositions \ref{Gevreyproduct} and \ref{Gevreycomposition} yield to
\[\left\Vert \partial ^ \alpha \overline{f_n}\right\Vert_V \leq \dfrac{c_hc_n}{y_n} \left( \dfrac{\rho_h}{y_n}\right)^{|\alpha|} (\alpha!)^{1+\gamma}, \] 
for every multi-index $\alpha \in \N ^4$ where $c_h = d^2c_ac_bc_g\rho_a\rho_b$ and 
\[ \rho _h = 36\max \left\lbrace \dfrac{4}{3}\rho_g(1+c_g),1+d\rho_a, 1+d\rho_b\right\rbrace. \]
Let $\delta_n = \frac{y_n}{c_hc_n}$. For any $\alpha \in \N^4$, $z_n \in \textup{Supp}\left(\overline{f}_n \right)$ 
\begin{align*}
|\partial^\alpha \overline{f}(z_n)| & \leq \epsilon_n c_a\left(\dfrac{\rho_h}{y_n}\right)^{|\alpha|}(\alpha !)^{1+\gamma} \\
& \xrightarrow[n \rightarrow \infty]{} 0 .
\end{align*} 
Thus $\overline{f}$ is infinitely differentiable on $V \cap \partial(\R^2_*\times \R^2_*)$ and it is actually $C^\infty$-flat on it. Hence $\overline{f}$ is of class $C^\infty$ on $V$. By Lemma \ref{estimate}, there exists $C > 0$, independent of $\alpha$ and $n$, such that \[ \dfrac{\epsilon_n}{y_n^{|\alpha|}} = \dfrac{e^{-1/y_n^{\rho}}}{y_n^{|\alpha|}} \leq C(\alpha!)^{\gamma}.\] 
Then 
\begin{align*}
\Vert \partial^\alpha \overline{f}\Vert_V & \leq \sup_{n \geq 1} \epsilon _n \delta_n\Vert (\partial ^{\alpha}\overline{h}_n)\Vert_V \\
& \leq C \rho_h^{|\alpha|} (\alpha !)^{1+2\gamma}\\ 
& = C \rho_h^{|\alpha|} (\alpha !)^{1+\sigma}.
\end{align*} 
Therefore $\overline{f} \in G^{1+\sigma}(U).$ Fix $\epsilon > 0$ and denote \[ H_\epsilon (\theta,R) = h(R) + \epsilon F(\theta,R). \]
 Notice that for $\theta^{(n)}$ such that $\sin(\theta^{(n)} \cdot k_n^{\bot}) = 1$ and every $R = (R_1, R_2) \in \Lambda_{y_n} \cap U$ with $R_1 \geq 2y_n$ we have 
\[ X_{H_\epsilon}(\theta^{(n)}, R) = (\nabla h (R), \epsilon\delta_n\epsilon_n k_n^\bot).\] 
Let $R^{(n)} = \Lambda_{y_n} \cap (\lbrace 2y_n \rbrace \times \R)$, which will belong to $U$ for $n$ sufficiently large. If $R^{(n)}$ is in $U$ and the first coordinate of $k_n^\bot$ is positive (resp. negative), the solution $z^{(n)}(t)$ starting at $(\theta^{(n)},R^{(n)})$ will satisfy 
 \begin{equation} 
 \label{ellipticlinearspeed}
 z_2^{(n)}(t) = R^{(n)}+ t\epsilon\epsilon_n \delta_nk_n^\bot
 \end{equation}
for all positive (resp. negative) $t$ where the flow is defined. Notice that for this values of $t$ the solutions will be contained on $\T^2 \times \R^2_+$. Therefore these unstable orbits will be taken by $T$ to orbits of $\overline{H}_\epsilon = \overline{h} + \epsilon \overline{f}$. This finishes the proof.
\end{proof}

\section{Appendix: Results on Gevrey-smooth functions}
\label{app: gevrey}

Gevrey functions were first introduced by Maurice Gevrey while studying solutions of partial differential equations. In \cite{gevrey_sur_1918} he proves that this class is actually closed under products and compositions. Here we state those results in our particular context. A proof in the more general context of ultra-differentiable functions can be found in \cite{rainer_composition_2014}.

\begin{prop}
\label{Gevreyproduct}
Let $f,$ $g: U \subset \mathbb {R}^d \rightarrow \mathbb{R}$ smooth and $K \subset U$ compact. Suppose 
\[\sup_{z \in K} |\partial ^ \alpha g(z)| \leq c_g\rho_g^{|\alpha|}(\alpha!)^s,\]
\[\sup_{z \in K} |\partial ^ \alpha f(z)| \leq c_f\rho_f^{|\alpha|}(\alpha!)^s,\]
 for every $\alpha \in \mathbb {N}^d$ and some positive constants $c_f,$ $c_g,$ $\rho_f,$ $\rho_g$. Then $h = f g :U \rightarrow \R$ satisfies 
\[ \sup_{z \in K} |\partial ^ \alpha h(z)| \leq c_fc_g(6\max \lbrace \rho_f,\rho_g\rbrace)^{|\alpha|}(\alpha!)^s \]
for every $\alpha \in \N^d$. 
\end{prop}

\begin{prop}
\label{Gevreycomposition}
Let $f: V \subset \mathbb {R}^m \rightarrow \mathbb{R}$, $g: U \subset \mathbb {R}^d \rightarrow V \subset \mathbb{R}^m $ smooth and $K_1 \subset U$, $K_2 \subset V $ compact sets with $g (K_1) \subset K_2$. Suppose 
\[ \sup_{z \in K_1} |\partial ^ \alpha g(z)| \leq c_g\rho_g^{|\alpha|}(\alpha!)^s,\]
\[ \sup_{z \in K_2} |\partial ^ \beta f(z)| \leq c_f\rho_f^{|\beta|}(\beta!)^s, \]
 for every $\alpha \in \mathbb {N}^d$, $\beta \in \mathbb{N}^m$ and some positive constants $c_f,c_g,\rho_f,\rho_g$. Then the composite function $h = f \circ g :U \rightarrow \R$ satisfies 
\[ \sup_{z \in K_1} |\partial ^ \alpha h(z)| \leq c_fc_g\rho_f(\rho_g(1+\rho_f c_g))^{|\alpha|}(\alpha!)^s \] 
for every $\alpha \in \N^d$. 
\end{prop}

\bibliographystyle{AIMS.bst}
\bibliography{Unstable.bib}

\end{document}